\documentclass[12pt]{article}
\usepackage{amsmath}
\usepackage{amsthm}
\usepackage{amssymb,amsfonts}
\usepackage[T1]{fontenc}
\usepackage{parskip}
\usepackage{hyperref}
\usepackage{latexsym}
\usepackage{todonotes}
\usepackage{upgreek}
\usepackage{rotating}
\usepackage{nicefrac}
\usepackage{epsfig}
\usepackage{stmaryrd}
\usepackage{setspace}
\usepackage{enumerate}
\usepackage[all]{xypic}
\usepackage{bbm,ifpdf,tikz}
\ifpdf
\usepackage{pdfsync}
\fi
\usepackage[alphabetic,nobysame]{amsrefs}

\newtheorem{theorem}{Theorem}[section]

\newtheorem{lemma}[theorem]{Lemma}
\newtheorem{proposition}[theorem]{Proposition}
\newtheorem{corollary}[theorem]{Corollary}

{
\theoremstyle{definition}
\newtheorem{definition}[theorem]{Definition}
\newtheorem{example}[theorem]{Example}

\newtheorem{remark}[theorem]{Remark}
}

\newcommand{\excise}[1]{}

\renewcommand{\dim}{\operatorname{dim}}

\renewcommand{\and}{\qquad\text{and}\qquad}

\newcommand{\HD}{\operatorname{HD}}
\newcommand{\rank}{\operatorname{rank}}

\newcommand{\Z}{\mathbb{Z}}
\newcommand{\Q}{\mathbb{Q}}

\newcommand{\R}{\mathbb{R}}

\newcommand{\cG}{\mathcal{G}}

\newcommand{\cGg}{\cG_{g,S}}

\newcommand{\cGgop}{\cGg^{\op}}

\newcommand{\cGop}{\mathcal{G}^{\op}}

\newcommand{\op}{{\operatorname{op}}}

\newcommand{\Spl}{\operatorname{Spl}}

\newcommand{\bp}{{\mathbf{p}}}

\begin{document}
\spacing{1.2}
\noindent{\LARGE\bf Universality theorems for generalized splines}\\

\noindent{\bf Jacob P. Matherne}\footnote{Supported by NSF Grant DMS-2452179 and Simons Foundation Travel Support for Mathematicians Award MPS-TSM-00007970.}\\
Department of Mathematics, North Carolina State University, Raleigh, NC
\vspace{.1in}

\noindent{\bf Eric Ramos}\footnote{Supported by NSF grants DMS-2137628 and DMS-2452031.}\\
Department of Mathematical Sciences, Stevens Institute of Technology, Hoboken, NJ
\vspace{.1in}

\noindent{\bf Julianna Tymoczko}\footnote{Supported by NSF grants DMS-1800773 and DMS-2054513}.\\
Department of Mathematics and Statistics, Smith College, Northampton, MA\\

{\small
\begin{quote}
\noindent {\em Abstract.}
We study generalized splines from the perspective of the representation theory of the category of graphs with contractions. Our main theorem proves a kind of finite generation, which in turn implies the existence of a ``universal generating set'' for the module of splines over any graph with fixed combinatorial genus. This theorem holds over any Noetherian commutative ring with a chosen finite list of ideals for edge-labels.  We then give several applications of this theorem, including showing that a particular generating function associated to splines on trees is algebraic when the base ring satisfies certain finiteness conditions. We illustrate our technical theorems explicitly by giving a classification of splines on graphs with combinatorial genus one and two.  
\end{quote} }

\section{Introduction}
The study of splines has been a mainstay in computational mathematics and algebraic geometry since the 1940's.  For general treatments, we point to \cite{S07,LS07}, for an emphasis on splines in computer graphics and design to \cite{BBB87}, and from the point of view of data interpolation to \cite{dB01}. In 2016, Gilbert, Tymoczko, and Viel developed a theory of splines (called generalized splines) on arbitrary graphs \cite{GTV}, generalizing work of Billera \cite{B88} and Guillemin--Zara \cite{GZ00,GZ01,GZ01b}. Parallel to this, there has  been an explosion of interest in recent years in the representation theory of categories, especially those categories which maintain a certain combinatorial flavor \cite{sam}. More specifically, there has been a recent push to categorify classical graph structure theorems, and apply these categorifications to a wealth of applications in algebra, topology, and combinatorics \cite{Barter,PR-trees,PR-genus,MR,KR}.

This paper unites these two threads in the literature by applying the representation theory of categories of graphs with edge-contractions \cite{PR-genus} to the study of generalized splines \cite{GTV}, providing a new framework with which to unify and expand upon results in the literature.  Our main results (described in Sections~\ref{subsec:intronoeth} and~\ref{sec:gensplinesintro} below) can be thought of as categorifying the theory of splines, explaining why the same  themes and constructions appear in a variety of different contexts in spline analysis  \cite{BT,GTV,PSTW,RS,AMT}, as well as:
\begin{itemize}
\item providing an intuitive sense of how bases of spaces of splines depend on different parameters of the graph (see, for example, Remark \ref{flow-up});
  \item developing asymptotic information about how splines on certain families of graphs grow (see Theorem \ref{algebraic}); 
  \item giving concrete generators for splines in special cases (see the examples of Section \ref{sec:smallComputations} and Theorem \ref{theorem: splines on graphs genus one and two}); and 
  \item establishing that (combinatorial) genus of graphs is a key parameter for spline analysis (see Sections \ref{section: algebraicity of trees} and \ref{examples}).
 \end{itemize}
Furthermore, our main theorems rely only on the coefficient ring being Noetherian and so apply broadly across all typical applications of splines.

The main results of this paper are not explicitly constructive.  However, they indicate where one could expect to find patterns and possibly even formulas.  For instance, these results suggest that to analyze the dimension of spaces of splines, one should look at families of graphs with constant genus. 

We now describe our main results with more technical precision.

\subsection{The category $\cGg$ and Noetherianity}\label{subsec:intronoeth}

For the remainder of this paper, we fix a commutative Noetherian ring $R$ with unit along with a finite set $S$ of ideals in $R$. An \textbf{edge-labeled graph} is a pair $(G,\alpha)$ where $G$ is a connected graph with finite vertex set $V_G$ and finite edge set $E_G$, and  $\alpha$ is a map $\alpha\colon E_G \rightarrow S$.  A \textbf{contraction} between edge-labeled graphs is a map  $\varphi\colon(G,\alpha) \rightarrow (G',\alpha')$  that replaces an edge $uv$ with a vertex $v'$ while preserving the other edges along with their labels (see Definition~\ref{def:contraction}). We write $\cGg$ for the category whose objects are edge-labeled graphs with \textbf{(combinatorial) genus} $g :=|E_G|-|V_G|+1$, and whose morphisms are contractions. Our primary interest is to study \textbf{representations} of the opposite category $\cGgop$.

Let $R$ be a commutative Noetherian ring with unit. A \textbf{$\cGgop$-module over $R$}---or an \textbf{$R$-representation} of $\cGgop$---is a covariant functor from $\cGgop$ to the category of finitely generated $R$-modules.  Intuitively, a \textbf{$\cGgop$-module} $M$ consists of a  collection of finitely generated $R$-modules $\{M(G,\alpha)\}_{(G,\alpha)}$, together with an assignment of an $R$-module homomorphism $\varphi^*\colon M(G',\alpha') \rightarrow M(G,\alpha)$ to each contraction of edge-labeled graphs $\varphi\colon(G,\alpha) \rightarrow (G',\alpha')$. 

A $\cGgop$-module $M$ is \textbf{finitely generated} if there exists a finite collection of edge-labeled graphs $\{(G_i,\alpha_i)\}_i$ such that for any edge-labeled graph $(G,\alpha)$, the $R$-module $M(G,\alpha)$ is spanned by the images of the $R$-modules $M(G_i,\alpha_i)$ under the maps induced by contractions, in the following sense.  Each $(G_i,\alpha_i)$ is obtained from $(G,\alpha)$ by contracting some subset of edges, forgetting the labels on all contracted edges, and preserving edge-labels on all remaining edges.  The $\cGgop$-module structure produces a set of $R$-module homomorphisms $\varphi^{\ast}_i \colon M(G_i,\alpha_i) \rightarrow M(G,\alpha)$.  Finite-generation means that $M(G,\alpha)$ is contained in the span of the $R$-modules $\varphi^{\ast}_i(M(G_i,\alpha_i))$. In this case we refer to $\{(G_i,\alpha_i)\}_i$ as the set of \textbf{generators} of $M$.  

Our primary technical theorem is the following.

\begin{theorem}\label{noetherian}
If $M$ is a finitely generated $\cGgop$-module, then all submodules of $M$ are also finitely generated.
\end{theorem}

In different language, Theorem \ref{noetherian} states that the module theory of the category $\cGgop$ is \textbf{locally Noetherian}; i.e., all submodules of finitely generated modules must themselves be finitely generated.

This result and its proof are similar to the result \cite[Theorem 1.1]{PR-genus}. The only modification is the addition of edge-labels to the graph. With this modification, however, we will be able to apply this theorem to the case of generalized splines.

Philosophically, Theorem~\ref{noetherian} means a $\cGgop$-module $M$ that is finitely generated has a kind of uniformity in the presentations of its modules $M(G,\alpha)$.  We further elucidate this recurrent theme in Sections~\ref{section: algebraicity of trees} and~\ref{examples} by undergoing a more extensive study of splines on graphs of genus $0$, $1$, and $2$.

\subsection{Generalized splines and main results}\label{sec:gensplinesintro}

Let $(G,\alpha)$ be an edge-labeled graph. The \textbf{module of generalized splines} $\Spl_R(G,\alpha)$ is the $R$-submodule of $\bigoplus_{v \in V_G} R$ consisting of all elements
$\bp$ such that $\bp_u - \bp_v \in \alpha(\{u,v\})$ whenever $\{u,v\}$ is an edge. The module of generalized splines was introduced in \cite{GTV}, aiming to generalize the theory of classical splines (see Definition \ref{def:classicalSplines}), which has a long and varied history (see \cite{LST} for a recent survey).  Much of the work with generalized splines has been concerned with the case where $R$ is a polynomial ring over a field and the edge-label set $S$ consists of principal ideals (see e.g. \cite[Section 6]{AMT}).

Focusing on this setting is largely motivated by the classical picture of splines, where one studies piecewise functions on polygonal decompositions of a subspace of $\mathbb{R}^n$ with some differentiability conditions on the boundaries of the polygons. In addition, the same ring and edge-label set also arises naturally in topology, where generalized splines describe the torus-equivariant cohomology of algebraic varieties satisfying appropriate conditions \cite{GKM,Tym05,T16b,GTV}.

The (many) special algebraic properties of polynomial rings and principal ideals are used extensively---if sometimes implicitly---in those classical applications. One important observation about the results in this paper is that they apply whenever $R$ is \emph{any} Noetherian ring. This stands in stark contrast to various results that suggest that the form of (or even the ability to compute) the minimal generating sets for the module of splines is very sensitive to the choice of ring $R$.  (See, e.g., \cite{BT} for examples of how the behavior when $R = \Z/n \Z$ differs significantly from the case when $R$ is an integral domain.)

Before we state our primary result, we observe that if $\varphi\colon(G,\alpha) \rightarrow (G',\alpha')$ is a contraction of edge-labeled graphs, then one obtains a natural morphism of modules
\[
\varphi^\ast \colon \Spl_R(G',\alpha') \rightarrow \Spl_R(G,\alpha)
\]
by sending a spline $({\bf{p}}_v)_{v \in V_{G'}}$ to the spline $({\bf{p}}_{\varphi(w)})_{w \in V_G}$.
We call the latter spline a \textbf{vertex expansion} of the former.  (See Figures~\ref{contractionexample} and~\ref{vertexexpansion} for an example of a contraction together with the induced morphism of spline modules.) These morphisms turn the assignment $(G,\alpha) \mapsto \Spl_R(G,\alpha)$ into a $\cGgop$-module which we denote by $\Spl_R$. The object $\Spl_R$ is a single object that encodes all possible generalized splines across all edge-labeled graphs of combinatorial genus $g$ with edge-labels in $S$. In other words, $\Spl_R$ is a kind of categorification of spline analysis. Our primary result is as follows.

\begin{theorem}\label{splinefg}
Let $R$ be a commutative Noetherian ring with unit, $S$ be a finite set of ideals in $R$, and $g \geq 0$ be a fixed integer. Then the $\cGgop$-module $\Spl_R$ is finitely generated. In particular, there exists a finite set $\{(G_i,\alpha_i)\}_i$ of edge-labeled graphs of genus $g$, depending only on $R, S$, and $g$, such that for any edge-labeled graph $(G,\alpha)$ of genus $g$, the module of generalized splines $\Spl_R(G,\alpha)$ is generated by the images of the modules $\Spl_R(G_i,\alpha_i)$ under the maps induced by contraction.
\end{theorem}

For the sake of concreteness, one may rephrase this theorem to read as follows.

\begin{theorem}\label{splinefgConcrete}
Let $R$ be a commutative Noetherian ring with unit, $S$ be a finite set of ideals of $R$, and $g \geq 0$ be a fixed integer. Then there exists a finite set of edge-labeled graphs $\{(G_i,\alpha_i)\}_i$ of genus $g$  such that for any choice $\{\beta_{i,j}\}_j$ of generating set of each $\Spl_R(G_i,\alpha_i)$, the following holds: If $(G,\alpha)$ is any edge-labeled graph of genus $g$, then the module of splines $\Spl_R(G,\alpha)$ is generated by vertex expansions of the splines $\beta_{i,j}$.
\end{theorem}

We note that Theorems~\ref{splinefg} and~\ref{splinefgConcrete} are \emph{not constructive}; that is, they do not produce the finite collection of edge-labeled graphs $\{(G_i,\alpha_i)\}_i$.  Nevertheless, they assert the existence of such a finite collection. In simple examples, it is possible to produce this collection: as a first example, consider the situation where $S = \{I\}$ consists of a single ideal in an arbitrary commutative Noetherian ring $R$ with unit. In this case, a generalized spline on an edge-labeled graph $(G,\alpha)$ is a tuple $({\bp}_v)_{v \in V_G}$ of elements of $R$ such that $\bp_u$ and $\bp_v$ are equal modulo $I$ whenever $\{u,v\}$ is an edge. Fix a (finite) generating set $x_1,\ldots,x_r$ for $I$. It has been shown in \cite[Theorem 10]{AMT} that $\Spl_R(G,\alpha)$ is generated by the all-ones vector $\mathbf{1}$, along with the indicator vectors which assign $x_i$ to the $u$-th coordinate and $0$ to all other coordinates. It follows from this that the $\cGgop$-module $\Spl_R$ must be generated by the rose, consisting of one vertex and $g$ loops, along with the (finitely many) genus $g$ graphs that have precisely two vertices.  See Sections \ref{section: algebraicity of trees} and \ref{section: genus one and two} for more involved examples and applications.

We end this section by highlighting some applications of Theorem~\ref{splinefgConcrete} (compare these to the bulleted list directly preceding Section~\ref{subsec:intronoeth} of the introduction):
\begin{itemize}
\item \textbf{An inductive construction of flow-up generating sets:} Given a total order $<$ on the vertices of an edge-labeled graph $(G,\alpha)$, we say that a spline $\bp$ is \textbf{flow-up} at $v$ if there is some vertex $v$ of $G$ such that $\bp_w = 0$ for all $w < v$. Flow-up generating sets for the module of splines are analogous to upper-triangular bases, and thus extremely useful for computational purposes. It is not known when one can find a minimum generating set of flow-up splines---i.e. a \textbf{flow-up basis}---though it is known to be true when $R = \mathbb{Z}/n\mathbb{Z}$ \cite{PSTW}. Our machinery constructs a version of universal flow-up generators (see Remark \ref{flow-up})---though it says nothing about minimality of these generators. 
\item \textbf{Algebraicity of the generating function for splines on trees over rings meeting finiteness constraints:} Literature on the representation theory of categories shows that---provided the combinatorics of the category is nice enough---one can often prove theorems about the types of growth permitted in the dimensions of the associated modules. In \cite{RamTree}, the second author showed that a particular generating function associated to representations of the category of rooted trees and contractions must be algebraic, in the sense that it satisfies a finite polynomial equation over $\Q(n)$. In this work we expand the theorem of the second author to show that the dimension growth of tree splines over certain finite and finite dimensional rings must also be algebraic in nature (see Theorem \ref{algebraic}). 
\item \textbf{Classification of splines on graphs of genus 0, 1, and 2:} We explicitly construct generators in the cases of genus zero (Section~\ref{section: algebraicity of trees}) and in genus one and two (Section \ref{section: genus one and two}).
\end{itemize}

\begin{remark}
In the recent work \cite{RS}, Rose and Suzuki consider splines over Euclidean domains. They construct universal flow-up splines that generate all graphs by considering a finite set of generating graphs obtained through a sequence of graph-theoretic operations on the originating graph. Only one of these operations involves contraction, the others being more purely combinatorial in nature. After the publication of \cite{RS}, Dilaver and Alt{\i}nok \cite{D25} provided an entirely different method for generating flow up bases over the integers. These works are related to Theorem \ref{splinefgConcrete} in spirit, though neither our results nor theirs imply the other. We summarize some differences:

\begin{itemize}
\item \cite{RS} works over Euclidean domains, whereas we consider splines over arbitrary Noetherian rings.
\item The methodology of \cite{RS} allows them to work with all graphs simultaneously, while we always fix the genus of the graphs being considered.
\item The techniques of \cite{RS} do not involve the use of representation stability or of representations of categories. As a consequence, they do not obtain results on growths of dimensions as we do.
\end{itemize}
\end{remark}

\section{Gr\"obner categories}\label{sec:grob}

In this section we review the necessary background on Sam and Snowden's theory of Gr\"obner categories \cite{sam}. We will ultimately use this theory to prove our main technical theorem, Theorem~\ref{noetherian}. Much of the exposition that follows closely mirrors \cite{PR-genus}.

We fix, for the remainder of this paper, a commutative Noetherian ring $R$ with unit, along with a finite set of ideals $S$ of $R$.

\subsection{Contractions of edge-labeled graphs}

\begin{definition}\label{def:contraction}

An \textbf{edge-labeled graph} is a pair $(G,\alpha)$, where $G =(V_G,E_G,\partial)$ is a triple comprised of two finite sets $V_G$ and $E_G$---the \textbf{vertex} and \textbf{edge} sets, respectively---and $\partial \colon E_G \sqcup V_G \rightarrow  \binom{V_G}{2} \sqcup V_G$ is the boundary map that sends each edge to its endpoints and each vertex to itself, and $\alpha \colon E_G \rightarrow S$ is a map of sets (see Figure \ref{labeleduncontractedgraph}).

\begin{figure}
    \centering
    \includegraphics{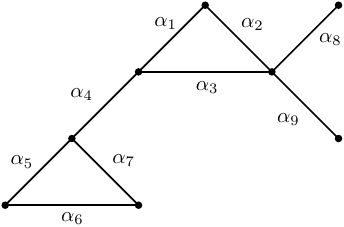}
    \caption{An example of an edge-labeled graph. The $\alpha_i$ live in the fixed finite set $S$ of ideals of our ring $R$.}
    \label{labeleduncontractedgraph}
\end{figure}

Those edges $e \in E_G$ with $\partial(e) \in V_G$ are known as \textbf{loops}. We exclusively consider graphs that are \textbf{connected}, in that for any two vertices $a,b \in V_G$ there exists a sequence of vertices $a = a_0,a_1,\dots, a_{n-1},a_n = b$ such that for all $0 \le i \le n-1$, we have $\{a_i,a_{i+1}\} = \partial(e_i)$ for some edge $e_i$. The \textbf{(combinatorial) genus} of a graph $G$ is the quantity $g := |E_G| - |V_G| + 1$. A graph of genus $0$ is known as a \textbf{tree}. 

Given two edge-labeled graphs $(G,\alpha),(G',\alpha')$, a \textbf{contraction} from $(G,\alpha)$ to $(G',\alpha')$ is a map
\[
\varphi \colon V_G \sqcup E_G \rightarrow V_{G'} \sqcup E_{G'}
\]
satisfying the following conditions:

\begin{enumerate}
\item $\varphi(V_G) = V_{G'}$ (i.e. vertices map to vertices, and this assignment is surjective);
\item if $e \in E_G$, then $\varphi(\partial(e))  = \partial(\varphi(e))$ (i.e. the map preserves adjacency);
\item the restriction $\varphi|_{\varphi^{-1}(E_{G'})}$ is a bijection, and $\alpha' \circ \varphi|_{\varphi^{-1}(E_{G'})}(e) = \alpha(e)$ for all $e \in \varphi^{-1}(E_{G'})$ (i.e. the edges of $G$ that are not sent to $V_{G'}$ are in bijective correspondence with the edges of $G'$ and carry the same edge-label);
\item for every vertex $v \in V_{G'}$, the preimage $\varphi^{-1}(v)$ is a subtree of $G$ (i.e. one never contracts a cycle).
\end{enumerate}
\end{definition}

If $\varphi$ is a contraction, we will refer to those edges of its domain that are sent to vertices as the \textbf{contracted edges}. An example of a contraction is depicted in Figure~\ref{contractionexample}: note that the final condition above implies that a contraction $\varphi$ may contract more than one edge, so long as a cycle is not contracted. We also note that graph automorphisms, in the usual sense, are contractions according to the above definition. 

\begin{figure}
\centering
\tikz[remember
picture]{\node(1BL){\includegraphics{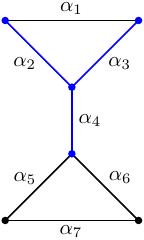}};}%
\hspace*{3cm}%
\tikz[remember picture]{\node(1BR){\includegraphics{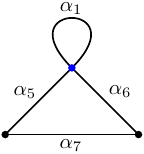}};}

\caption{A contraction of a labeled graph of genus $2$. The contracted edges are colored blue.}

\label{contractionexample}

\end{figure}

\subsection{Graph categories}

We now develop edge-labeled versions of the categories in \cite{PR-trees,PR-genus}. We  write $\cGg$ for the category whose objects are edge-labeled graphs of genus $g$ and whose morphisms are contractions. An \textbf{edge-labeled rooted graph $(G,T,v,\alpha)$} is an edge-labeled graph $(G,\alpha)$ together with a choice of rooted spanning tree $(T,v)$ (see Figure~\ref{rigidexamples} left).

A contraction between edge-labeled rooted graphs $\varphi \colon (G,T,v,\alpha) \rightarrow (G',T',v',\alpha')$ is a contraction between the underlying edge-labeled graphs with the added conditions that
\begin{enumerate}
\item $\varphi(v) = v'$ (i.e. the root is sent to the root);
\item for any edge $e \in E_G$, if $\varphi(e) \in V_{G'}$ then $e$ is an edge of $T$ (i.e. one is only permitted to contract edges within the provided spanning tree).
\end{enumerate}

The category of edge-labeled rooted graphs of genus $g$ with contractions is denoted $R\cGg$. 

Finally, we will choose additional structure that both allows us to treat $G$ as if it were drawn with a planar embedding, and breaks endomorphisms. Note that each edge in a rooted tree $(T,v)$ can be unambiguously directed towards the endpoint farther from the root. Using these edge-directions, we can recreate a notion of ``left-to-right" for not-necessarily-planar graphs: at each vertex $u \in T$, choose a total ordering $\sigma_v$ of the edges in $T$ directed out of $u$.  This data further allows us to use a depth-first algorithm to define a total ordering $<$ on the vertices of $T$, with $\sigma_u$ determining which order to explore the edges leaving each vertex $u$.

Putting this together, a \textbf{rigidified edge-labeled graph $(G,T,v,\alpha,\sigma,\tau)$} is an edge-labeled rooted graph equipped with a collection of total orderings $\sigma = \{\sigma_u\}_{u \in T}$,  one for each vertex $u$ of $T$, of the edges out of $u$ when the rooted tree $T$ is directed as above, as well as a total ordering $\tau$ of the edges outside the spanning tree $T$, also called the \textbf{extra edges} (see Figure~\ref{rigidexamples} right). A contraction of rigidified edge-labeled graphs is a contraction of the underlying edge-labeled rooted graphs with the extra conditions that
\begin{enumerate}
\item for any extra edge $e\in E_G$, we have $\tau(e) = \tau'(\varphi(e))$ (i.e. the total ordering on the extra edges is \emph{preserved} under contraction);
\item for any two vertices $v,w \in V_G$, if $v < w$ in the depth-first total order with edge-order determined by $\sigma_u$ at each vertex $u$, then $\varphi(v) < \varphi(w)$ (i.e. contractions respect the depth-first total order on vertices).
\end{enumerate}

We write $P\cGg$ for the category of rigidified edge-labeled graphs of genus equal to $g$ with contractions as morphisms.

\begin{figure}
    \centering
    \includegraphics{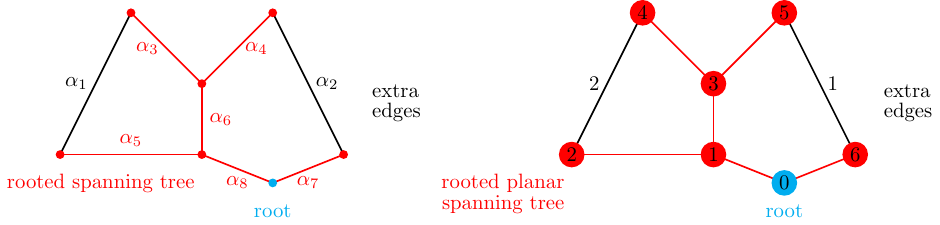}
    \caption{Examples of a rooted graph of genus $2$, as well as that same graph rigidified, respectively. The edge-labels have been omitted from the rigidified graph for clarity. The numbers next to the extra edges in the rigidified graph indicate their internal order $\tau$. {\color{red} In this case, the ordering $\sigma_u$ at each vertex is ``left to right" in the planar embedding of the rooted spanning tree, so the vertices are labeled by the usual depth-first search from the root.}}
    \label{rigidexamples}
\end{figure}

The category $P\cGg$ will be used as a technical tool for proving Theorem \ref{noetherian}. In particular, rigidified edge-labeled graphs have been designed to disallow endomorphisms (recall that graph automorphisms appear in $\cGg$), as it is impossible to map in a nontrivial way from a rigidified edge-labeled graph to itself while preserving all necessary structures. This is critical for the Sam--Snowden Gr\"obner category theory which we use in the proof of Theorem \ref{noetherian}. Because our set of permissible labels $S$ is finite, the proof of Theorem \ref{noetherian} will be almost identical to the proof appearing in \cite[Theorem 1.1]{PR-genus} of Noetherianity in the unlabeled case. As such, we will oftentimes direct the reader to that work for a more detailed explication of certain steps.

\begin{definition}\label{def:modulefggens}
If $\mathcal{C}$ is any one of the above categories, and $R$ is a commutative Noetherian ring with unity, then a \textbf{$\mathcal{C}^{\op}$-module over $R$} is a covariant functor $M\colon\mathcal{C}^{\op} \rightarrow R\text{-mod}$ from $\mathcal{C}^{\op}$ to the category of finitely generated $R$-modules. We say that a $\mathcal{C}^{\op}$-module is \textbf{finitely generated} if there exists a finite list of objects $\{G_i\}_i$ of $\mathcal{C}$ such that for any object $G$ of $\mathcal{C}$, the module $M(G)$ is spanned by the images of the $M(G_i)$ under the maps induced by contractions. We call the members of this finite list \textbf{generators} for $M$.
\end{definition}

The category of $\mathcal{C}^{\op}$-modules with natural transformations as morphisms is an abelian category with all abelian operations defined pointwise. In particular, one may consider \textbf{submodules} of a $\mathcal{C}^{\op}$-module. More concretely, a submodule $N$ of a $\mathcal{C}^{\op}$-module $M$ is a collection of $R$-submodules $N(G) \subseteq M(G)$, one for each object $G$ of $\mathcal{C}$, which the restrictions of the induced maps of $M$ preserve. Understanding the submodules of a finitely generated module is the major content of Theorem \ref{noetherian}.

We now collect a series of results that we will need for the proof of Theorem~\ref{noetherian}, which appears at the end of this subsection.

\begin{lemma}\label{PF}
If all submodules of finitely generated $P\cGgop$-modules are once again finitely generated, then the same is true of $R\cGgop$-modules and $\cGgop$-modules.
\end{lemma}

\begin{proof}
By how these categories are defined, it is immediate that the forgetful functors from $P\cGgop$-modules to $R\cGgop$-modules and to $\cGgop$-modules satisfy Property (F) in \cite[Definition 3.2.1]{sam}. This implies our desired conclusion.
\end{proof}

In light of the previous lemma, it is sufficient to consider modules over $P\cGop_{g,S}$. Indeed, to prove Theorem~\ref{noetherian} we will show that $P\cGop_{g,S}$ is \textbf{Gr\"obner} in the sense of \cite{sam}, which is strictly stronger than the statement that its module theory is (locally) Noetherian.

Proving that $P\cGop_{g,S}$ is Gr\"obner amounts to showing the following two facts about the morphisms of $P\cGop_{g,S}$:

\begin{itemize}
\item (G1) For any rigidified edge-labeled graph $(G,T,v,\alpha,\sigma,\tau)$, there exists a well order $\prec$ on the set $\mathcal{H}_{(G,T,v,\alpha,\sigma,\tau)}$ of isomorphism classes of morphisms with domain $(G,T,v,\alpha,\sigma,\tau)$ such that for any two elements $\varphi,\psi$ of $\mathcal{H}_{(G,T,v,\alpha,\sigma,\tau)}$ with $\varphi \prec \psi$, we have $h \circ \varphi \prec h \circ \psi$ for any morphism $h$ for which these compositions make sense.
\item (G2) For any rigidified edge-labeled graph $(G,T,v,\alpha,\sigma,\tau)$, the poset  $(\mathcal{H}_{(G,T,v,\alpha,\sigma,\tau)},\leq)$, where $\leq$ is the natural factor-through order defined by $\varphi \leq \psi$ if and only if there is some $\eta$ such that $\eta \circ \varphi = \psi$ does not admit infinite descending chains or infinite anti-chains.
\end{itemize}

As the name suggests, Gr\"obner theory is heavily influenced by the classical notion of Gr\"obner bases in polynomial rings. In that more classical theory, one must choose a (non-canonical) monomial order. The ability to make such a choice in our setting is the essential content of condition (G1). In the classical setting, once one has used this monomial order to reduce the question of Noetherianity to proving that all monomial ideals are finitely generated, one then argues that the set of all monomials forms a poset under the (canonical) divisibility order, and that the ideals of this poset are all finitely generated. The ability to perform this step is the essential content of condition (G2).

At this point of our exposition, we have imposed a number of orders on a number of different objects. We therefore take just a moment to collect all of them in one place for the reader's convenience.

\begin{itemize}
    \item Each rigidified graph $(G,T,v,\alpha,\sigma,\tau)$ is equipped with the data of a rooted spanning tree (and thus directions on the edges of $T$) and a collection $\sigma = \{\sigma_u\}_{u \in T}$ that chooses, for each vertex $u$, a total ordering $\sigma_u$ of the edges in $T$ directed out of $u$ (to mimic ``left-to-right"). Using the ordering of edges $\sigma_u$ at each vertex $u$, we can define a depth-first search on the vertex set of the graph, thus obtaining a total order on $V_G$. (In other words, begin at the root, and walk along the edges of the spanning tree, choosing the minimum available amongst $\sigma_u$ at each turn.) The depth-first order on the vertex set is usually the most relevant feature of all of this data. For instance, the definition of morphisms in the category $P\cGgop$ requires that this depth-first order be respected, in that if one vertex is below another in the order, then the image of the smaller vertex must be smaller than the image of the larger vertex.
    \item The rigidified graph also contains the data of a total ordering on its edges that are not in the spanning tree (i.e. the extra edges). Because we only work in circumstances where the genus is a fixed integer $g$, rigidified graphs will always have precisely $g$ extra edges. Morphisms in the category of rigidified graphs are required to \emph{preserve} this ordering, in that each extra edge must map to an extra edge with exactly the same label. Both of these first two orders exist primarily to remove symmetry (e.g. endomorphisms) from the category of edge-labeled graphs, while also not losing too much information from the morphisms of $\cGgop$. In other words, every morphism of $\cGgop$ appears \emph{somewhere} in $P\cGgop$, it is just a matter of choosing the correct data for the domain and codomain graphs to find it.
    \item Having fixed a rigidified graph $(G,T,v,\alpha,\sigma,\tau)$, we have two important orderings on the set of isomorphism classes of maps with domain $(G,T,v,\alpha,\sigma,\tau)$. The first of these orderings is a non-canonical well order that we denote $\prec$. We require that $\prec$ respect composition in that if $\varphi \prec \psi$, 
    then $h \circ \varphi \prec h \circ \psi$, for any morphism $h$ in $P\cGgop$ for which the composition makes sense. It is not immediately obvious that such an order $\prec$ exists. We will prove this in Proposition \ref{prop:G1}.
    \item The second ordering on isomorphism classes of morphisms originating from a fixed $(G,T,v,\alpha,\sigma,\tau)$ is the canonical divisibility order, which we denote by $\leq$. In this order, $\varphi \leq \psi$ if and only if there exists some morphism $f$ in $P\cGgop$ such that $f \circ \varphi = \psi$. Note that, by working in $P\cGgop$ and thereby eliminating non-trivial endomorphisms, the divisibility order becomes a poset order.
\end{itemize}

\begin{proposition}\label{prop:G1}
The category $P\cGgop$ satisfies property (G1).
\end{proposition}

\begin{proof}
Choose a total order on the set of isomorphism classes of objects of $P\cGgop$, and denote it $\prec_O$. Such an order can be chosen assuming the axiom of choice. Fix a rigidified edge-labeled graph $(G,T,v,\alpha,\sigma,\tau)$. Now, for any isomorphism classes of morphisms $\varphi\colon(G,T,v,\alpha,\sigma,\tau) \rightarrow (G',T',v',\alpha',\sigma',\tau')$ and $\psi\colon(G,T,v,\alpha,\sigma,\tau) \rightarrow (G'',T'',v'',\alpha'',\sigma'',\tau'')$, we declare $\varphi \prec \psi$ if:
\begin{itemize}
\item $(G',T',v',\alpha',\sigma',\tau') \prec_O (G'',T'',v'',\alpha'',\sigma'',\tau'')$ or they are equal, and
\item $\varphi^{\op}(w) < \psi^{\op}(w)$ in the depth-first order, where $w$ is the smallest (in the depth-first order) vertex of $T' = T''$ for which $\varphi^{\op}(w) \neq \psi^{\op}(w)$.
\end{itemize}

To see that $\prec$ respects composition, note that for composition to make sense the two morphisms $\varphi$ and $\psi$ must have equal codomain. Therefore, assuming that $\varphi \prec \psi$, it must be the case that the second condition of the above has been met. This second condition will remain true following composition with any morphisms, as our morphisms preserve the depth-first order by definition.
\end{proof}

The property (G2) is considerably more delicate, and relies heavily on the results of \cite{PR-genus}.

\begin{proposition}
The category $P\cGgop$ satisfies property (G2).
\end{proposition}

\begin{proof}
In \cite[Corollary 3.9]{PR-genus} it is proven that the category of rigidified genus $g$ graphs, without any edge-labels, satisfies (G2). That proof functions on the back of the labeled planar Kruskal's tree theorem (see \cite[Theorem 9]{Barter} for the unlabeled version and \cite[Theorem 3.6]{PR-genus} for the labeled version). The category we are working with here differs from the category of that work by the presence of a finite amount of extra data (namely the set $S$). This data can be incorporated into the argument of \cite[Corollary 3.9]{PR-genus} in the following way. When they associate to each morphism a vertex-labeled planar rooted tree, instead of the portion of the label indicating the presence of an extra edge being a $1$ or $0$, we use the edge-label of that edge or $0$. Moreover, we also add a label to every vertex indicating the label of the edge of the tree going into it. Because the edge-label set was chosen to be finite, this extra data is still permissible under the assumptions of the planar Kruskal's tree theorem. In particular, the argument of \cite[Corollary 3.9]{PR-genus} goes through without issue, giving us what we need.
\end{proof}

We finish this section with the proof of Theorem~\ref{noetherian}.

\begin{proof}[Proof of Theorem \ref{noetherian}]
We have thus-far proven that $P\cGgop$ is a directed (i.e. without non-trivial endomorphisms) category which satisfies properties (G1) and (G2). In the language of \cite{sam}, this implies that $P\cGgop$ is a Gr\"obner category;  therefore, modules over $P\cGgop$ satisfy the Noetherian property. Lemma \ref{PF} now implies our result.
\end{proof}

\section{Contraction in generalized splines}\label{SplineAsModule}

The goal of this section is to apply the machinery of Theorem~\ref{noetherian} to the setting of generalized splines: this will lead to the proof of Theorem~\ref{splinefg}, which appears at the end of this section.

\begin{definition}
Fix a commutative Noetherian ring $R$ with unity, and let $(G,\alpha)$ be an edge-labeled graph. The \textbf{module of generalized splines} over $(G,\alpha)$ is the $R$-submodule of $\bigoplus_{v \in V_G}R$ defined by
\[
\Spl_R(G,\alpha) = \left\{{\bp} \in \bigoplus_{v \in V_G}R \ \middle\vert\ {\bp}_u - {\bp}_v \in \alpha(e) \text{ for all edges $e\in E_G$ with $\partial(e) = \{u,v\}$}\right\}.
\]
\end{definition}

We now point out how the module of generalized splines behaves when contracting edges. It was observed in \cite[Lemma 23 and Corollary 27]{AMT} and reiterated in Section~\ref{sec:gensplinesintro} that if $\varphi\colon (G,\alpha) \rightarrow (G',\alpha')$ is a contraction of edge-labeled graphs, then one obtains a ``vertex expansion'' map
\[
\varphi^{\ast}\colon \Spl_R(G',\alpha') \rightarrow \Spl_R(G,\alpha)
\]
by setting
\begin{align}
\varphi^{\ast}(({\bf{p}}_v)_{v \in V_{G'}}) = ({\bf{p}}_{\varphi(w)})_{w \in V_G}. \label{contmap}
\end{align}

In particular, this turns the assignment
$(G,\alpha) \mapsto \Spl_R(G,\alpha)$
into a $\cGgop$-module over $R$ that we denote by $\Spl_R$. (See Figure~\ref{vertexexpansion} for an example of the induced map on spline modules stemming from the contraction in Figure \ref{contractionexample}.) Our next result proves that the module $\Spl_R$ is finitely generated.

\begin{figure}
\centering
\tikz[remember
picture]{\node(1BL){\includegraphics[]{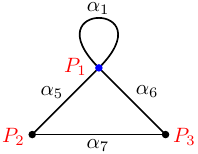}};}%
\hspace*{3cm}%
\tikz[remember picture]{\node(1BR){\includegraphics[]{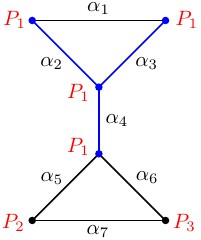}};}

\caption{An example of the map on splines induced by the contraction of Figure \ref{contractionexample}.}

\label{vertexexpansion}

\end{figure}
\tikz[overlay,remember picture]{\draw[-latex,thick] (1BL) -- (1BL-|1BR.west)
node[midway,below,text width=2.5cm]{};} 

\begin{proof}[Proof of Theorem \ref{splinefg}]
Define a $\cGgop$-module $M$ by associating to any edge-labeled graph $(G,\alpha)$ of genus $g$, with edge-labels in $S$, the $R$-module $M(G,\alpha):=\bigoplus_{v \in V_G} R$ and to any contraction $\varphi\colon (G,\alpha) \rightarrow (G',\alpha')$ the induced map as in \eqref{contmap}. Since $\Spl_R$ is a submodule of $M$, it will suffice by Theorem~\ref{noetherian} to show that $M$ is finitely generated.

Let $(G,\alpha)$ be an edge-labeled graph with at least three vertices, and write ${\mathbf{1}^u}\in \bigoplus_{v \in V_G} R$ for the indicator vector defined by 
\[
({\mathbf{1}^u})_v = \begin{cases}1 & \text{if $v = u$}\\ 0 & \text{if $v \neq u$}.\end{cases}
\]
We will find ${\mathbf{1}}^{u}$ in the span of images of maps induced by contractions coming from graphs with strictly fewer vertices. Assume first that we may find a non-loop edge $e$ for which $u \notin \partial(e)$. Let $\varphi$ be a contraction from $(G,\alpha)$ whose only contracted edge is $e$. Then it is clear that ${\mathbf{1}^u}$ is in the image of the map $\varphi^\ast$. On the other hand, if $G$ has a vertex $u$ that is adjacent to every non-loop edge, it must be the case that $G$ is a star graph (perhaps with some loops on vertices and some multi-edges). In this case, let $v$ be any vertex adjacent to $u$. The vector ${\mathbf{1}^v}$ can be found in the image of the map induced by the one that contracts all non-loop edges of $G$ that are not connected to $v$.   On the other hand, the vector ${\mathbf{1}^u} + {\mathbf{1}^v}$ can be found in the image of the map induced by the one that only contracts one of the edges connecting $u$ to $v$. In all cases, it therefore follows that for any $G$ with at least three vertices, each of the vectors ${\mathbf{1}^u}$ will appear in the span of the images of maps induced by contractions coming from graphs with strictly fewer vertices. Because the ${\mathbf{1}^u}$ generate $M(G,\alpha) = \bigoplus_{v \in V_G} R$ as an $R$-module, and because there are only a finite number of edge-labeled graphs with at most two vertices, we conclude that $M$ is generated by the edge-labeled graphs with at most two vertices, as desired.
\end{proof}

\begin{remark}\label{flow-up}
Theorem \ref{splinefg} implies that the $P\cGgop$-module $\Spl_R$ is also finitely generated. The extra data imposed on a rigidified edge-labeled graph imposes a total ordering on the vertices via a depth-first ordering from the root of the chosen spanning tree. Moreover, the morphisms in the category are designed to preserve this order on the vertices. By consequence, if one starts with a \textbf{flow-up spline}, i.e. a spline such that there exists a vertex $v$ where all vertices above $v$ are assigned $0$ by the spline, then the vertex expansion will still be flow-up. This implies that if one finds a generating set of flow-up splines for each of the modules of splines associated to the generating graphs of $\Spl_R$, then one can find flow-up generating sets for all modules of splines that are uniform in some sense. Unfortunately, this will not allow one to conclude that generalized splines always have a \textbf{flow-up basis}, i.e. a minimum generating set of flow-up splines, as for general rings one cannot always find a basis within a generating set.   In some cases, however, positive results on the existence of flow-up bases are known: for a couple of examples in this direction, see e.g. \cite[Theorem 6.14]{D25} for the existence of flow-up bases for splines over $\Z$ and \cite[Section 5]{BT} for the existence of flow-up bases for splines on cycles over $\Z/n\Z$.
\end{remark}

\section{Application: Splines over trees} \label{section: algebraicity of trees}

In the various subsections that follow, we will study splines on trees, and each section will be devoted to a different coefficient ring.

\subsection{Application: classical splines over trees}

The classical theory of splines involves the study of piecewise polynomial functions  on a combinatorially-defined partition of a geometric object (usually a polyhedral complex), whose polynomial pieces agree up to some degree of differentiability on the intersection of the top-dimensional pieces of the partition.  Due to a result of Billera \cite[Theorem 2.4]{B88}, this classical theory of splines may be studied using the theory of generalized splines over a certain quotient of a polynomial ring.

\begin{definition}\label{def:classicalSplines}
Let $k$ be a field, and write $R_d^n = k[x_1,\ldots,x_n] / \mathfrak{m}^d$ for the polynomial ring over $k$ modulo the $d$-th power of the ideal $\mathfrak{m}$  generated by the variables $x_i$. For any edge-labeled graph $(G,\alpha)$, we will refer to the module of splines $\Spl_{R_d^n}(G,\alpha)$ over $R_d^n$ as a \textbf{module of classical splines}.
\end{definition}

For more about the relationship between $\Spl_{R^n_d}$ and the classical theory of splines, we point to e.g. \cite[Sections 5 and 6]{AMT} for a detailed exposition.  Note that $R_d^n$ is a finite-dimensional $k$-algebra, and therefore any module of classical splines is a finite-dimensional $k$-vector space. It is of great and classical interest to understand the dimensions of the modules of classical splines as a $k$-vector space, see e.g. \cite{S74,S79,AS87,AS90,D91,SS02,LS07,SSY20}.  And, similar problems about the size of various modules of generalized splines have been studied in the literature, see e.g. \cite{BR91,GZ01,GZ03,BHKR,GTV,ACFMT,AS21a,AS21b}.

The main theorem of this section (Theorem~\ref{thm:treealgebraic} below) provides results in this direction for the specific case of classical splines over trees. We point out that a presentation for modules of splines over trees is given in \cite{GTV}: this presentation is useful for advancing the theory, but is difficult to use for answering questions about dimensions of the spline modules.  

Our application to classical splines over trees will naturally involve $P\mathcal{G}_{0,S}^{\op}$-modules.  Thus, we first begin with the following general lemma.

\begin{lemma}\label{lem:restrictscalars}
Let $k$ be a field, and let $R$ be any finite-dimensional $k$-algebra. Then for any finite set $S$ of ideals of $R$, the restriction of scalars functor from $P\mathcal{G}_{0,S}^{\op}$-modules over $R$ to $P\mathcal{G}_{0,S}^{\op}$-modules over $k$ preserves finite generation.
\end{lemma}

\begin{proof}
This follows immediately from the fact that $R$ is finite dimensional over $k$. Any generating set over $R$ can be turned into a generating set over $k$ by adding in all $R$-translates of the generators. 
\end{proof}

This lemma allows us to use facts about dimension growth of $P\mathcal{G}_{0,S}^{\op}$-modules over a field to conclude similar facts in the context of classical splines. It remains to study the dimension growth of $P\mathcal{G}_{0,S}^{\op}$-modules over $k$.

\begin{definition}\label{words}
Let $S$ be any finite set. The \textbf{Dyck language} $\mathcal{D}(S)$ over $S$ is the language on the alphabet
\[
\{\, (_\alpha\, ,\, )_\alpha \mid \alpha \in S\},
\]
where nonempty words consist of a sequence of properly nested parentheses labeled by elements of $S$. By properly nested, we mean that any left parenthesis must be later closed-off by a right parenthesis with the same label.  For example, if $S = \{0,1\}$, then $(_1)_1$ and $(_1(_0)_0)_1$ are words in $\mathcal{D}(S)$, while $(_1)_0$ is not.

Next, let $S$ be a finite set of ideals of some commutative ring $R$, and consider the Dyck language $\mathcal{D}(S)$. For any word $w \in \mathcal{D}(S)$, we can define a planar rooted tree with edge-labels in $S$ as follows. (See Figure \ref{dycktree} for an example.) By definition, we set the tree associated to the empty word to be the single-noded planar rooted tree. For any nonempty word $w \in \mathcal{D}(S)$, we may find (possibly empty, possibly repeating) elements $w_1,w_2,w_3,\ldots, w_n \in \mathcal{D}(S)$ and $\alpha_1,\ldots,\alpha_n \in S$ such that  
\[
w = (_{\alpha_1} w_1 )_{\alpha_1} (_{\alpha_2} w_2 )_{\alpha_2} \cdots (_{\alpha_n} w_n )_{\alpha_n}.
\] 

Then the associated edge-labeled tree $T(w)$ is defined by having $n$ branches off its root, with edge-labels $\alpha_1,\ldots,\alpha_n$, respectively, where for each $i$ the $i$-th branch from the left is attached to the tree defined by the word $w_i$. It is easily checked that this defines a bijection between $\mathcal{D}(S)$ and the set of planar rooted trees with edge-labels in $S$. 

\begin{figure}
    \centering
    \includegraphics{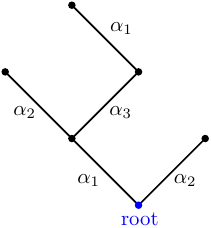}
    \caption{The labeled planar rooted tree associated to the word $(_{\alpha_1}(_{\alpha_2})_{\alpha_2}(_{\alpha_3}(_{\alpha_1})_{\alpha_1})_{\alpha_3})_{\alpha_1}(_{\alpha_2})_{\alpha_2}$}
    \label{dycktree}
\end{figure}

For any $P\mathcal{G}_{0,S}^{\op}$-module $M$ over a field $k$, we define the \textbf{Hilbert--Dyck series} as the formal power series
\[
\HD_M(t) := \sum_{w \in \mathcal{D}(S)}\dim_k M(T(w)) t^{|E_{T(w)}|},
\]
where $|E_{T(w)}|$ is the number of edges in $T(w)$, or equivalently, half the length of the word $w$.  
\end{definition}

For example, if $M$ is the module that assigns $k$ to every tree and the identity to every morphism, and if $S = \{(0)\}$, then $\HD_M(t)$ is the generating function for the Catalan numbers. It is well-known that this generating function is \textbf{algebraic}, in that there exists a polynomial $P(x,t)$ such that $P(\HD_M(t),t) = 0$. Algebraic power series are critical objects in analytic combinatorics (see \cite{BD} for a survey on algebraic generating functions and their properties) and impose a number of strong restrictions on the growth of their coefficients. For instance, for a sequence of integers $\{a_n\}$ whose generating function $\sum_{n}a_nt^n$ is algebraic, there exist constants $C,\rho \in \R$ and a rational number $\alpha$ such that $a_n$ is asymptotically equal to $Cn^\alpha\rho^n$,
\[
a_n \sim C n^\alpha\rho^{n}.
\]

\begin{theorem} \label{algebraic}
Let $M$ be a finitely generated $P\mathcal{G}_{0,S}^{\op}$-module over a field $k$. Then the Hilbert--Dyck series $\HD_M(t)$ is algebraic.
\end{theorem}

\begin{proof}[Sketch of proof]
    In \cite{RamTree}, the second author extended the classical result on the algebraicity of the generating function for the Catalan numbers to say that the Hilbert--Dyck series of any finitely-generated $P\mathcal{G}_{0}^{\op}$-module over a field must be algebraic.  The proof of that theorem used a connection between the representation theory of categories and language theory that was originally observed in \cite{sam}; in particular, explicit push-down automata were constructed whose associated languages have generating functions that agree with the above Hilbert--Dyck series. In a similar fashion to the proof of Theorem \ref{noetherian}, the main theorem of \cite{RamTree} can be altered to accommodate the finite bit of extra data coming from $S$.
\end{proof}

Via Lemma~\ref{lem:restrictscalars} from the beginning of this section, we can now conclude the following about the dimensions of classical splines.

\begin{theorem} \label{thm:treealgebraic}
Let $k$ be a field, and let $S$ denote any finite set of ideals of $R_d^n$. The Hilbert--Dyck series
\[
\HD_d(t) := \sum_{w \in \mathcal{D}(S)} \dim_k(\Spl_{R_d^n}(T(w)))t^{|E_{T(w)}|}
\]
is algebraic.  
\end{theorem}

\subsection{Application: splines over trees for finite rings of prime power order}

We now explore the machinery of this paper when the ring $R$ is finite. The benefit in this setting is that the set of all non-zero ideals in $R$ is finite, and therefore there is no limitation for the set of ideals used as edge-labels. 

To this point in the literature, splines over finite rings have been limited to the most natural setting of $R = \Z/n\Z$. Much of the work done on generalized splines over $\mathbb{Z}/n\mathbb{Z}$ has been focused on computing explicit bases \cite{BHKR,BT}: this was ultimately accomplished in full generality in \cite{PSTW}. Rather than rehash these well-studied situations, in this work we instead  focus our attention on the class of finite rings whose order is a prime power. Rings of prime power order play an important role in number theory, and their enumeration has been a problem of interest for decades \cite{KP,Poonen,BM}. The current best asymptotic estimates imply that the number of rings of prime power order $p^n$ is exponential in $\frac{2}{27}n^3 + O(n^{\frac{5}{2}})$ \cite{BM}.

\begin{definition}
Let $M$ be an $R$-module for a finite ring $R$.  A \textbf{basis} of $M$ is a generating set for $M$ consisting of the fewest number of elements.  A \textbf{$\mathbb{Z}$-basis} of $M$ is a set of elements of the smallest possible size that generates $M$ when viewed as a $\mathbb{Z}$-module (in the natural way, via the additive structure of $R$).
\end{definition}

\begin{theorem} \label{primePower}
Let $R$ be a finite ring whose order is a power of some prime number $p$, and let $S$ denote its set of ideals. Then the generating function
\[
\HD(t) := \sum_{w \in \mathcal{D}(S)} \rank_{\mathbb{Z}}(\Spl_R(T(w)))t^{|E_{T(w)}|}
\]
is algebraic, where $\rank_\mathbb{Z}(\Spl_R(T(w)))$ denotes the size of a $\mathbb{Z}$-basis for $\Spl_R(T(w))$.
\end{theorem}

\begin{proof}
Since $R$ is a finite ring of prime power order, the structure theorem for finite abelian groups asserts that as an additive group $\Spl_R(T(w))$ must have the decomposition
\[
\Spl_R(T(w)) \cong \bigoplus_{n \geq 1} m_{n,w} \mathbb{Z}/p^n\mathbb{Z},
\]
where $m_{n,w} = 0$ for all $n \gg 0$ and all words $w \in \mathcal{D}(S)$. It follows that
\[
 \rank_{\mathbb{Z}}(\Spl_R(T(w))) = \sum_{n \geq 1} m_{n,w}.
 \]
 On the other hand, the vector space $\Spl_R(T(w)) \otimes_\mathbb{Z} \mathbb{Z}/p\mathbb{Z}$ can easily be seen to have dimension $\sum_{n \geq 1} m_{n,w}$. Our claim now follows from Theorem \ref{algebraic}, as well as the fact that extension of scalars preserves finite generation of the $P\mathcal{G}_{0,S}^{\op}$-module $\Spl_R$.
 \end{proof}
 
 \begin{remark}
 Theorem \ref{primePower} provides a partial answer to \cite[Question 6.3]{BT}. That question asks whether one can determine the distributions of integral ranks of modules of splines over prime power rings by varying the edge labels. Our result implies that these ranks fit into an algebraic generating function, which implies the shape of their growth must satisfy a very particular form. 
 \end{remark}

 \begin{remark}
If $R$ is a finite ring that is not of prime power order, then as a $\mathbb{Z}$-module  the module of splines can be decomposed into a sum $\bigoplus_i \mathbb{Z}/d_i\mathbb{Z}$, where $d_i \mid d_{i+1}$. In particular, it follows from the Chinese remainder theorem that the $\mathbb{Z}$-rank of this module is related to the total number of prime powers that appear for each prime. To be more precise, if we write $n_p$ for the total number of powers of $p$ that appear as maximal $p$ divisors of the $d_i$, then the $\mathbb{Z}$-rank of the module of splines is precisely given by $\max_{p \text{ prime}}n_p$. While the above techniques show that the generating function for each $n_p$ is algebraic, this does not necessarily translate into saying that the maximum is as well.
\end{remark}

\subsection{Application: splines over trees for $\Z/p^n\Z$} \label{sec:smallComputations}

We now specialize to splines of trees over $R = \Z/p^n\Z$ for some fixed prime $p$.  For the cases when $n = 1,2$, we explicitly compute the Hilbert--Dyck series and therefore concretely illustrate the algebraicity results of the previous sections.

\begin{example}
 We first consider generalized splines of trees over the field $R=\Z/p\Z$. Since $R$ is a field, we take $S = \{0,R\}$ to be our set of ideals. It therefore follows that for any edge-labeled tree $(T,\alpha)$ the module of splines $\Spl_R(T,\alpha)$ is generated in the following way: for each maximal subtree $T'$ of $T$ whose every edge is labeled with a $0$, one has a generator $\mathbf{1}^{T'}$, the indicator function for the vertices of $T'$. Notice that for each vertex $v$ whose every adjacent edge is labeled with $R$, one has a generator $\mathbf{1}^v$, the indicator function for $v$.  

This reveals to us two things about the $\mathcal{G}_{0,S}^{\op}$ -module $\Spl_R$. Firstly, one can see from this that $\Spl_R$ is generated by the single vertex tree, as well as the single edge tree with edge-label $R$. 

Secondly, being that we are working over a field, the integral rank of $\Spl_R(T,\alpha)$ is equal to the dimension of this vector space over $R$, which the above computation has shown is equal to the number of components in the spanning subforest of $T$ generated by the edges of $T$ with label zero.  That is, it is one more than the number of edge-labels equal to $R$. In other words, writing $n_0$ for the number of edges with label 0, we have 
\[
\dim_R \Spl_R(T,\alpha) = |E_T|-n_0 + 1.
\]

It follows that the power series
\[
\sum_{w \in \mathcal{D}(S)}(|E_{T(w)}|-n_0 + 1) t^{|E_{T(w)}|}
\]
must be algebraic. 

In this case, it is also possible to deduce this algebraicity from more elementary analytic combinatorics.  To this end, notice that
\begin{align*}
\sum_{w \in \mathcal{D}(S)}(|E_{T(w)}|-n_0 + 1) t^{|E_{T(w)}|} = \sum_{w \in \mathcal{D}(S)}|E_{T(w)}| t^{|E_{T(w)}|}-\sum_{w \in \mathcal{D}(S)}n_0 t^{|E_{T(w)}|} + \sum_{w \in \mathcal{D}(S)}t^{|E_{T(w)}|}.
\end{align*}

The first and third power series on the right hand side are clearly algebraic, as the third power series is the generating function for the unambiguous context-free language $\mathcal{D}(S)$, and the first is $t$ times the derivative of an algebraic power series. The middle series is a bit trickier, but can be seen to be algebraic by the use of an auxiliary language. Let $\mathcal{D}(S)'$ be the language on the alphabet $(_0 ,(_R, )_0, )_R, (_0^\ast, )_0^\ast$ built in the following way: the words in this language are almost identical to those in $\mathcal{D}(S)$, except that the symbol pair of $(_0^\ast$ and $)_0^\ast$ must appear exactly once.  It can be shown that this language is unambiguous and context free by modifying the usual push down automaton for $\mathcal{D}(S)$ to include an extra state which keeps track of whether an edge has been marked yet or not.  One may think of a word in $\mathcal{D}(S)$ as being an edge-labeled planar rooted tree, but with exactly one of its edges labeled zero being marked in some way. In particular, for every word in $\mathcal{D}(S)$ whose corresponding edge-labeled tree is $T$, there are precisely $n_0$ words in $\mathcal{D}(S)$ whose associated tree is also $T$, given by choosing which edge labeled $0$ is to be marked. Then the generating function (with norm only counting left parenthesis) for this language is
\[
\sum_{w \in \mathcal{D}(S')}t^{|w|/2} = \sum_{w \in \mathcal{D}(S)}n_0t^{|w|/2}.
\]
\end{example}

\begin{example}
We next consider the case of $n = 2$, i.e. splines of trees over $R = \Z/p^2\Z$. This case was essentially completed in \cite[Theorem 5.2]{BT}. In that work it was shown that if all edge-labels are the unique non-trivial proper ideal $(p)$, then the integral rank of the module of splines is precisely the number of vertices of $T$. In particular, if we had chosen to limit our edge-labels to only include $(p)$, our Hilbert--Dyck series becomes the clearly algebraic
\[
\sum_{w \in \mathcal{D}(S)} (|w|/2 + 1)t^{|w|/2}.
\]

In the cases where we allow for the $0$ ideal to be an edge-label, one notes that the integral rank of splines over a tree $T$ is equal to that of the integral rank of splines over the tree obtained from $T$ by contracting all of the edges whose label is $0$. This observation therefore lands us back in the previous case. More specifically, if $w$ is a word in $\mathcal{D}(S)$ and $n_0$ is the number of letters in this word that use the label 0, then the integral rank of $\Spl_R(T(w))$ is precisely
\[
|E_{T(w)}| - n_0 + 1.
\]
This generating function was proven to be algebraic in the previous example.
\end{example}

The combinatorics of the Hilbert--Dyck series becomes considerably more complicated as one works with more complex rings such as $\Z/p^n\Z$ for $n > 2$. Even using the computation techniques of \cite{BT}, explicit computation for larger $n$ seems difficult. Despite this fact, Theorem \ref{primePower} guarantees that the integral rank of the module of splines in these settings must grow and behave in a  controlled way.

\section{Application: the genus one and two cases}\label{examples}

Having focused on the genus $0$ case (trees) in Section~\ref{section: algebraicity of trees}, we will now focus on the genus $1$ and $2$ cases.  The strategies developed previously in the paper allow us in many circumstances to recursively identify spline spaces and create bases, especially flow-up bases.  In Section~\ref{section: surjective quotient map and sections} we look at cases where the contraction maps described earlier in this paper give a \emph{section} in the topological sense, namely a map from a quotient to the original space whose composition with the quotient map is the identity.  We then extend this analysis to another class of contraction maps.  In Section~\ref{section: genus one and two} we use this to classify splines on all genus one and genus two graphs.  In all of Section~\ref{examples}, we take $R$ to be any commutative ring with identity.

\subsection{Restrictions of splines} \label{section: surjective quotient map and sections}

Section \ref{SplineAsModule} described a family of contraction maps on graphs $\varphi\colon G \rightarrow G'$ that induce maps on splines $\varphi^*\colon \Spl_R(G',\alpha') \rightarrow \Spl_R(G,\alpha)$.  The map $\varphi^*$ is generally not a bijection and cannot be inverted.  But if $G'$ is a particular kind of subgraph of $G$ then restricting ${\bf p} \in \Spl_R(G,\alpha)$ to the vertices in $G'$ gives a well-defined map on splines. In this case the composition $\rho \circ \varphi^*$ is the identity.

Motivated by constructions in topology, we say that $G$ is the \textit{one-point union} of two graphs $G_1$ and $G_2$ at vertex $v$, and write $G = G_1 \vee G_2$, if $V_{G_1} \cap V_{G_2} = \{v\}$ and the only edges $uu'$ with $u \in V_{G_1}$ and $u' \in V_{G_2}$ occur when one of $u, u'$ is $v$. In this case, each graph $G_i$ can be realized as both a subgraph of $G_1 \vee G_2$ and the image of a contraction of $G_1 \vee G_2$, if we relax the fourth condition in the definition of contraction (see Definition~\ref{def:contraction}). Note that in the notation $G_1 \vee G_2$, we have suppressed mention of the vertex $v$ to avoid cumbersome notation; however, in what follows we sometimes write $G_1 \vee_v G_2$ when it is important to emphasize it.


If $(G,\alpha)$ is any edge-labeled graph and $U \subseteq V_{G}$ is any subset of vertices in $G$, then we define the submodule of splines \textit{based at $U$} as
    \[\Spl_R(G,\alpha;U) = \{ {\bf p} \in \Spl_R(G,\alpha) \mid {\bf p}_v = 0 \text{ for all } v \in U\}. \] 
    It was shown in \cite[Lemma 2.8]{ACFMT} that for any edge-labeled graph $(G,\alpha)$ and vertex $v \in V_G$ we have 
\begin{equation} \label{equation: based splines}
    \Spl_R(G,\alpha) = R{\bf 1} \oplus \Spl_R(G,\alpha; v)
\end{equation}

Using this, we can express splines on a one-point union as follows.

\begin{lemma} \label{lemma: one-point union splines}
    Suppose $(G_1 \vee G_2, \alpha)$ is an edge-labeled graph that is a one-point union of two subgraphs $(G_1,\alpha_1), (G_2,\alpha_2)$ at $v$, where each $\alpha_i = \alpha|_{E_{G_i}}$. For each $i$, 
    \begin{itemize}
        \item let $\varphi_i\colon G_1 \vee G_2 \rightarrow G_i$ denote the generalized contraction obtained by sending all vertices outside of $V_{G_i}$ to $v$, and
        \item let $\rho_i\colon \Spl_R(G_1 \vee G_2, \alpha) \rightarrow \Spl_R(G_i, \alpha_i)$ denote the restriction of splines from the one-point union to $G_i$.
    \end{itemize}
    Then for each $i$ the restriction map $\rho_i$ is surjective and $\varphi_i^*$ is a section of the map, in the sense that $\rho_i \circ \varphi_i^*$ is the identity. Equivalently, we have the following:
   \[\begin{array}{ll}
    \Spl_R(G_1 \vee G_2,\alpha) &= R {\bf 1} \oplus \varphi_1^*(\Spl_R(G_1,\alpha_1; v)) \oplus \varphi_2^*(\Spl_R(G_2,\alpha_2; v)) \\
    &= \varphi_1^*(\Spl_R(G_1,\alpha_1)) \oplus \varphi_2^*(\Spl_R(G_2,\alpha_2; v)) \\
    &= \varphi_1^*(\Spl_R(G_1,\alpha_1; v)) \oplus \varphi_2^*(\Spl_R(G_2,\alpha_2)). 
    \end{array}\]
\end{lemma}
    
    \begin{proof} 
    Since both $\varphi_i^*$ and $\rho_i$ are defined to be the identity on all vertices in $G_i$ we know the composition $\rho_i \circ \varphi_i^*$ is the identity for each $i$.  For each $i$ we have
    \[\varphi_i^*(\Spl_R(G_i,\alpha_i)) \subseteq \Spl_R(G_1 \vee G_2, \alpha),\]
    so we conclude also that $\rho_i$ is surjective for each $i$.

    The kernel of the map $\rho_1$ is defined to be the collection of splines that are zero on all of the vertices of $G_1$.  In other words we have
    \[\ker \rho_1 = \varphi_2^*(\Spl_R(G_2,\alpha_2;v)),\]
    and similarly for $\rho_2$.  The rest of the claim follows from the First Isomorphism Theorem and Equation~\eqref{equation: based splines}.
    \end{proof}

    This leads to several corollaries whose proofs are immediate.  To start, observe that if $\Spl_R(G_1,\alpha_1)$ and $\Spl_R(G_2,\alpha_2)$ both have flow-up bases then the previous result produces a flow-up basis on $G_1 \vee G_2$. In other words, the proof of the following is immediate.

\begin{corollary} \label{corollary: flow-up bases}
    Suppose $(G_1 \vee G_2, \alpha)$ is a one-point union at $v$ with notation as in Lemma~\ref{lemma: one-point union splines}.  If $\mathcal{B}_i$ is a flow-up basis for each $\Spl_R(G_i,\alpha_i)$, the vertices of $G_1$ are all ordered before the vertices of $G_2$ (including $v$), and $\bf{p}^v$ denotes the flow-up basis element in $\mathcal{B}_2$ corresponding to vertex $v$ then 
    \[\varphi_1^*(\mathcal{B}_1) \cup \varphi_2^*(\mathcal{B}_2 - \{\bf{p}^v\}) \]
    is a flow-up basis for $\Spl_R(G_1 \vee G_2,\alpha)$. 
\end{corollary}

The previous corollary reconstructs the flow-up basis on trees from \cite[Section 4]{GTV} in the special case when $G_1 \vee G_2$ is a tree.  More generally, we obtain considerable information about the generating set for any graph obtained by repeatedly taking the one-point union of a graph with a sequence of trees.

    \begin{corollary}\label{corollary: wedging trees}
  Suppose that $G_1$ is a graph without any leaves, equivalently suppose all vertices in $G_1$ have degree at least two. Suppose $G_k$ is any graph formed recursively from $G_1$ by a one-point union $G_{i+1} = G_i \vee_{v_i} T_i$ with a tree $T_i$. Let $\alpha$ be an edge-labeling on $G_k$.  Then 
\[\Spl_R(G_k,\alpha) \cong \Spl_R(G_1,\alpha|_{G_1}) \oplus \bigoplus_{i=1}^{k-1} \Spl_R(T_i,\alpha|_{T_i}; v_i).\]
Suppose that $\mathcal{G}$ is a minimal generating set for $\Spl_R(G_1,\alpha|_{G_1})$, the vertices of $G_k$ are ordered so that all vertices in $T_i$ come after $v_i$, and $\mathcal{B}_i$ is a flow-up basis for each tree $T_i$.  Through slight abuse of notation, let $\varphi_k^*(\bf{p})$ denote the image of any spline on $G_1, T_1, \ldots, T_{k-1}$ in $\Spl_R(G_k,\alpha)$.  Then we get a minimal generating set for $\Spl_R(G_k,\alpha)$ by
\[\varphi_k^*(\mathcal{G}) \cup \bigcup_{i=1}^{k-1} \varphi_k^*\left(\mathcal{B}_i - \{\bf{p}^{v_i}\}\right)\]
($\bf{p}^{v_i}$ is defined analogously here as in Lemma~\ref{lemma: one-point union splines}), and this set is flow-up on $G_k$ whenever $\mathcal{G}$ is flow-up on $G_1$.
    \end{corollary}

    \begin{proof}
All of the claims follow inductively from Lemma~\ref{lemma: one-point union splines} and Corollary~\ref{corollary: flow-up bases}.
    \end{proof}

Recall that a \emph{bridge} is an edge whose removal disconnects a graph.  The next corollary analyzes splines on a graph that contains a path whose edges are all bridges.  It extends the result \cite[Corollary 2.11]{ACFMT}.

\begin{corollary} \label{corollary: bridges}
Suppose that $G, G'$ are two disjoint graphs containing vertices $u, u'$ respectively, and that $P_k = v_0 v_1 v_2 \cdots v_k$ is a path of length $k$ that is disjoint from both $G$ and $G'$ except $v_0=u$ and $v_k = u'$.  Let $\varphi\colon G \vee P_k \vee G' \rightarrow G$ contract all vertices outside of $V_G$ to $u$, and similarly let $\varphi'\colon G \vee P_k \vee G' \rightarrow G'$ contract all vertices outside $V_{G'}$ to $u'$.  Let $\varphi_k\colon G \vee P_k \vee G' \rightarrow P_k$ contract all vertices in $V_G$ to $u$ and all vertices in $V_{G'}$ to $u'$.  Then 
\[\Spl_R(G \vee P_k \vee G',\alpha) = \varphi^*(\Spl_R(G,\alpha|_{G})) \oplus \varphi_k^*(\Spl_R(P_k,\alpha_{P_k)}; v_0)) \oplus \varphi'^*(\Spl_R(G',\alpha|_{G'};u')).\]
Moreover, if vertices in $G$ are ordered first, followed by those in $P_k$, and then by those in $G'$, and if $\mathcal{B}, \mathcal{B}_k, \mathcal{B}'$ are generating sets for $\Spl_R(G,\alpha|_{G}), \Spl_R(P_k,\alpha|_{P_k}), \Spl_R(G',\alpha|_{G'}; u')$, respectively, with $\mathcal{B}_k$ flow-up, then
\[\varphi^*(\mathcal{B}) \cup \varphi_k^*(\mathcal{B}_k - \{\bf{p}^{v_0}\}) \cup \varphi'^*(\mathcal{B}')\]
is a generating set for $\Spl_R(G \vee P_k \vee G',\alpha)$ that is minimal (respectively flow-up) if both $\mathcal{B}$ and $\mathcal{B}'$ are.
\end{corollary}

\begin{proof}
First take $G_1 = G \vee P_k$ and $G_2 = G'$ in Lemma~\ref{lemma: one-point union splines} and then apply Lemma~\ref{lemma: one-point union splines} again with $G_1 = G$ and $G_2 = P_k$ to get the decomposition. Then apply Corollary~\ref{corollary: flow-up bases} to get the result about generators.  
\end{proof}

We end this section with a lemma that describes the effects of contracting a  path when exactly one edge-label is changed.

\begin{lemma} \label{lemma: contracting path}
Suppose that $G$ contains a path $P_k = v_0 v_1 \cdots v_k$ and that all of $v_1, v_2, \ldots, v_{k-1}$ have degree two in $G$.  Suppose that $G'$ is the graph with $V_{G'} = V_G - \{v_1, \ldots, v_{k-1}\}$ and $E_{G'} = (E_G \cup \{v_0v_k\}) - \{v_0v_1, v_1v_2, \ldots, v_{k-1}v_k\}$. Let $\alpha$ be an edge-labeling of $G$ and let $\alpha'$ be the edge-labeling of $G'$ defined by $\alpha'(e)=\alpha(e)$ for all $e \neq v_0v_k$ and by 
\[\alpha'(v_0v_k) = \alpha(v_0v_1) \oplus \alpha(v_1v_2) \oplus \cdots \oplus \alpha(v_{k-1}v_k).\]
Then the restriction map $\rho\colon \Spl_R(G,\alpha) \rightarrow \Spl_R(G',\alpha')$ is a surjection with kernel isomorphic to $\Spl_R(P_k, \alpha|_{P_k}; \{v_0, v_k\})$.  Furthermore the kernel satisfies
\[\Spl_R(P_k, \alpha|_{P_k}; \{v_0,v_k\}) \cong \Spl_R(C_k, \alpha|_{P_k}; v_0),\] where $C_k$ is the cycle graph with $k$ edges with the same edge-labeling as $P_k$.
\end{lemma}

\begin{proof}
First we show the restriction map $\rho$ is well-defined on splines. If ${\bf p} \in \Spl_R(G,\alpha)$ then 
\[{\bf p}_{v_k} - {\bf p}_{v_0} = \left({\bf p}_{v_k} - {\bf p}_{v_{k-1}}\right) + \left({\bf p}_{v_{k-1}} - {\bf p}_{v_{k-2}}\right) + \cdots + \left({\bf p}_{v_1} - {\bf p}_{v_{0}}\right) \in \alpha'(v_0v_k)\]
by construction of the labeling $\alpha'$.  Similarly, we prove surjectivity.  If ${\bf q} \in \Spl_R(G', \alpha')$ then there exist $r_i \in \alpha(v_{i-1}v_i)$ for each $i = 1, \ldots, k$ with 
\[{\bf q}_{v_k}-{\bf q}_{v_0} = r_1 + r_2 + \cdots + r_k.\]
Define a spline ${\bf p}$ by the rule that for each $i=1,\ldots, k-1$ we have 
\[{\bf p}_{v_i} = {\bf q}_{v_0}+r_1+r_2 + \cdots + r_i\]
and for all other $u \in V_{G'}$ we have ${\bf p}_u={\bf q}_u$.  By definition of $r_i$, we have ${\bf p}_{v_i}-{\bf p}_{v_{i-1}} = r_i$ for all vertices on the path $P_k$.  If $uu' \in E_{G}$ is any edge off the path $P_k$, then 
\[{\bf p}_u-{\bf p}_{u'} = {\bf q}_u-{\bf q}_{u'} \in \alpha(uu') \]
since $\alpha'$ and $\alpha$ agree on those edges.  Thus ${\bf p} \in \Spl_R(G,\alpha)$ and $\rho({\bf p}) = {\bf q}$.

The kernel of the map $\rho$ consists of those splines ${\bf p} \in \Spl_R(G,\alpha)$ with ${\bf p}_u=0$ whenever $u \in V_{G'}$, namely $\Spl_R(G, \alpha; V_{G'})$.  Since $V_{G'} = V_G - \{v_1, v_2, \ldots, v_{k-1}\}$ this based subspace of splines is isomorphic to the space of splines on the subgraph of $G$ that is induced by the path $P_k$ based at $\{v_0, v_k\}$.  In other words,
\[\Spl_R(G, \alpha; V_{G'}) \cong \Spl_R(P_k, \alpha_{P_k}; \{v_0, v_k\}).\]
Now consider the $k$-cycle $C_k$ obtained from $P_k$ by identifying the endpoints of the path, in other words the edge-labeled graph with vertex set and edge set written
\[V_{C_k} = \{v_1, v_2, \ldots, v_{k-1}, u=\{v_0, v_k\}\} \textup{ and } E_{C_k} = \{v_1v_2, v_2v_3, \ldots, v_{k-1}u, uv_1\}\]
and edge labels $\alpha$ agreeing with those on $P_k$.  Since all spline conditions agree, we have
\[\Spl_R(P_k, \alpha; \{v_0, v_k\}) = \Spl_R(C_k,\alpha; u).\]
The first isomorphism theorem now completes the proof.
\end{proof}

\subsection{Splines on graphs of genus one and two} \label{section: genus one and two}

In this final section of the work we aim to implicitly find the generators of the $\cGgop$-module of splines in genus one and two. While this section does not directly appeal to any results of prior sections, other than the definition of the map induced by edge contractions, it does illustrate morally \emph{why} our finite generation result holds at least in small genus.

Throughout this section we will largely be considering the \emph{reduced} graphs of genus one or two.  There is a finite list of reduced graphs for each genus, allowing us to classify fairly straightforwardly. Reduced graphs play an important role in topological work like that of Culler and Vogtmann on outer space \cite{CV,Vogt} and Chan, Galatius, and Payne on moduli spaces of tropical curves \cite{CGP}.  

We start by defining what it means to be reduced.

\begin{definition} \label{definition: reduced graph}
A graph with more than one vertex is {\textit{reduced}} if it has neither a vertex of degree $2$ nor a bridge, namely an edge whose removal disconnects the graph.  
\end{definition}

We next prove that there is a small (finite) list of graphs of genus one and two.

\begin{lemma}\label{lemma: classifying reduced graphs g = 1, 2}
If $G$ is a reduced graph of genus $1$ then $G$ consists of a single loop at a vertex, namely $V_G=\{v\}$ and $E_G=\{vv\}$.

If $G$ is a reduced graph of genus $2$ then $G$ consists either of the figure eight, namely $V_G=\{v\}$ and $E_G = \{e_1, e_2\}$ where both $e_i$ start and end at $v$, or $G$ consists of the theta graph, namely $V_G=\{u, v\}$ and $E_G=\{e_1, e_2, e_3\}$ where all three $e_i$ start at $u$ and end at $v$, as shown in Figure~\ref{figure: reduced genus 2}.
\end{lemma}

\begin{proof}
First recall the following formula:
\begin{equation}\label{equation: degree sum and edges}
\sum \deg v = 2 |E_G|,
\end{equation}
which is obtained by counting the edges in $G$ incident to each vertex and then noting that each edge is counted twice since both endpoints are vertices.  (This is true even for edges of the form $vv$ for some vertex $v$.)

If the genus is $g$ then by definition
\[|E_G|-|V_G|+1=g,\]
and so $|E_G|=|V_G| + g-1$. With Equation~\eqref{equation: degree sum and edges}, this gives $\sum_{v \in V_G} \deg v = 2|V_G| + 2g-2$.  If the graph is reduced and has more than one vertex then $\deg v \geq 3$ for all $v \in V_G$, so
\[2|V_G|+2g-2= \sum_{v \in V_G} \deg v \geq \sum_{v \in V_G} 3 = 3|V_G|.\]

If $g=1$ this is a contradiction, meaning the only reduced graph of genus one has a single vertex $V_G = \{v\}$ and a single edge $E_G = \{vv\}$.

If $g=2$ this simplifies to $2 \geq |V_G|$.  Now analyze the two cases.  If $|V_G| = 1$ then by the definition of genus we conclude $|E_G|=2$ and each edge must start and end at the same vertex.  If $|V_G|=2$ then $|E_G|=3$.  If both vertices have loops then either they have a bridge between them or the graph is disconnected, both of which contract our hypothesis.  If exactly one vertex has a loop, then then the other vertex has degree two or less, which contradicts the assumption that $G$ is reduced.  This gives the two cases shown in Figure~\ref{figure: reduced genus 2}.
\begin{figure}[h] \label{figure: reduced genus 2}
    \centering
    \includegraphics{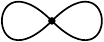}
    \hspace{3cm}
    \includegraphics{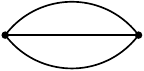}
    \caption{The two reduced graphs of genus 2: The figure eight and theta graph, respectively}
    \label{twogenus}
\end{figure}
\end{proof}

Moreover, the following result analyzes explicitly all possible contractions that preserve the genus of a graph.

\begin{lemma} \label{lemma: contractions that preserve genus}
If $G$ is a graph of genus $g$ then there is a sequence of contractions $\varphi_i$ from $G$ to a reduced graph of genus $g$ so that each $\varphi_i$ has one of the following three types:
\begin{enumerate}
    \item {\bf Contraction of a tree:} $G$ has an induced subgraph $T$ that is a rooted tree and the map $\varphi\colon G \rightarrow G'$ contracts all of $T$ to its root.
    \item {\bf Contraction of a bridge path:} $G$ has two vertices $u$ and $v$ with a path $P$ between them consisting entirely of bridge edges, and $\varphi\colon G \rightarrow G'$ contracts all of $P$ to the single vertex $u \sim v$.
    \item {\bf Contraction of a non-bridge path:} $G$ has two vertices $u$ and $v$ with a path $P$ between them and another path $P'$ between them that is disjoint from $P$ except at $u, v$, and $\varphi\colon G \rightarrow G'$ contracts $P$ to a single edge $uv$.
\end{enumerate}
\end{lemma}

\begin{proof}
Suppose $G$ is a graph with genus $g$.  If $u \in V_G$ has degree one then the induced subgraph consisting of the single edge $uv$ and its endpoints can be considered a tree, which we root at $v$.  The contraction of this leaf subgraph to $v$ does not change the genus of $G$ since we removed one edge and one vertex simultaneously, so the new genus is
\[ (|E_G|-1) - (|V_G|-1)+1 = |E_G| - |V_G|+1.\]
Contracting a rooted tree $T$ to its root can be interpreted as a sequence of contractions of these leaf subgraphs, so also do not change the genus.  By performing this first type of contractions, we may assume that $G$ does not have any vertices of degree one.

Now suppose that $G$ has a bridge, say the edge $uv$.  The contraction that sends the entire edge to either endpoint preserves genus since, like before, this contraction removes a single edge and single vertex simultaneously.  If $P$ is a path of length $k$ in $G$ whose edges are all bridges, then we can contract to either endpoint of $P$ via a sequence of single-edge contractions while preserving genus, as claimed. Moreover, these contractions do not create any vertices of degree one.  So we assume $G$ has neither vertices of degree one nor bridges.

Finally, suppose $P$ is a path in $G$ of length $k+1 \geq 2$ between vertices $u_0$ and $u_{k+1}$ (possibly equal) and suppose every interior vertex in $P$ has degree two.  Each vertex of degree two lies on such a path, e.g. the path formed by the vertex and its neighbors.  The contraction of $u_0u_1\cdots u_k$ to $u_0$ deletes exactly $k$ edges and $k$ vertices from $G$ so again preserves genus.  Moreover, this contraction eliminates vertices of degree two while creating neither vertices of degree one nor bridges.  This proves the claim.
\end{proof}

Combining this with the previous subsection gives the following.

\begin{theorem} \label{theorem: splines on graphs genus one and two}
Suppose $G$ is a graph of genus one or two and $G_r$ is a reduced graph of the same genus as $G$.  Choose contractions of the types in Lemma~\ref{lemma: contractions that preserve genus} as follows:
\begin{itemize}
    \item Contract all subtrees of $G$ with the composition $\varphi_{Tr}\colon G \rightarrow G_{Tr}$ of maps of type 1 from Lemma~\ref{lemma: contractions that preserve genus}, where $G_{Tr}$ has the same genus as $G$ but no vertices of degree one.  Let  $\mathcal{T}$ be the union of the subtrees of $G$ contracted by $\varphi_{Tr}$.
    \item Contract all bridges in $G_{Tr}$ with the composition $\varphi_{B}\colon G_{Tr} \rightarrow G_{B}$ of maps of type 2 from Lemma~\ref{lemma: contractions that preserve genus}, where $G_{B}$ has the same genus as $G$ but neither vertices of degree one or bridges.  Let $\mathcal{B}$ be the union of  all bridge paths contracted by $\varphi_{B}$.  Let $P$ denote a connected component of $\mathcal{B}$, namely a maximal bridge path in $G_{Tr}$.
    \item Contract all non-bridge paths with the composition $\varphi_2\colon G_{B} \rightarrow G_r$ of maps of type 3 from Lemma~\ref{lemma: contractions that preserve genus}, where $G_r$ is a reduced graph of the same genus as $G$.  Let $\mathcal{C}$ denote the union of the non-bridge paths contracted by $\varphi_2$.  Write $C \in \mathcal{C}$ to denote the cycle obtained from each maximal path in $\mathcal{C}$ by identifying endpoints.
\end{itemize}  

Denote the composition $\varphi\colon G \rightarrow G_r$.  Let $\alpha$ be an edge-labeling on $G$.  Then $\varphi$ induces a surjection on splines whose kernel is isomorphic to
\[\bigoplus_{T \in \mathcal{T}} \Spl_R(T, \alpha|_T; \varphi_{Tr}(T)) \oplus \bigoplus_{P \in \mathcal{B}} \Spl_R(P, \alpha|_P; \varphi_B(P)) \oplus \bigoplus_{C \in \mathcal{C}} \Spl_R(C, \alpha|_C; u_C \sim v_C)\]
where $u_C \sim v_C$ denotes the endpoints identified to create the cycle $C$.  Moreover, this map preserves flow-up generators except perhaps on the cycles.
\end{theorem}

\begin{proof}
Corollary~\ref{corollary: wedging trees} asserts that the map $\varphi_{Tr}$ induces a surjection from $\Spl_R(G,\alpha)$ onto $\Spl_R(G_{Tr},\alpha|_{G_{Tr}})$ with kernel $\bigoplus_{T \in \mathcal{T}} \Spl_R(T,\alpha|_T;\varphi_{Tr}(T))$.  

If $P$ is a bridge path in the graph $G_{Tr}$ then contracting $P$ produces the one-point union of the induced subgraphs of $G_{Tr}$ at the endpoints of $P$.  Using Lemma~\ref{lemma: one-point union splines} and Corollary~\ref{corollary: bridges}, we thus conclude that $\varphi_B$ induces a surjection from $\Spl_R(G_{Tr},\alpha|_{G_{Tr}})$ onto $\Spl_R(G_B,\alpha|_{G_B})$ with fiber $\bigoplus_{P \in \mathcal{B}} \Spl_R(P,\alpha|_P; \varphi_B(P))$ as claimed.

Finally, from Lemma~\ref{lemma: contracting path} we see that $\varphi_2$ induces a surjection from $\Spl_R(G_B,\alpha|_{G_B})$ onto $\Spl_R(G_r,\alpha|_{G_r})$ with the specified kernel.  Combining these gives the result, because the first two contraction maps preserve flow-up bases.
\end{proof}

Thus our methods provide an explicit computational tool for spline calculations.

\bibliography{./symplectic}
\bibliographystyle{amsalpha}

\end{document}